\documentclass[12pt,a4paper]{amsart}
\usepackage{amssymb,enumitem,mathtools}

\usepackage[breaklinks=true,colorlinks=true,linkcolor=blue,citecolor=red,urlcolor=green]{hyperref}

\usepackage[margin=1in,footskip=0.25in]{geometry}

\usepackage{color}

\def\mapsto{\longmapsto}

\def\cwedge{\bigcirc\kern-1.07em\wedge\ }

\newcommand \br{\mathbb{R}}

\newcommand \so{\mathfrak{so}}
\newcommand \ad{\operatorname{ad}}
\newcommand \<{\left<}
\renewcommand \>{\right>}
\newcommand \ip{\< \cdot, \cdot \>}

\newcommand \g{\mathfrak{g}}
\newcommand \gv{\mathfrak{v}}
\newcommand \gl{\mathfrak{l}}
\newcommand \gn{\mathfrak{n}}
\newcommand \gz{\mathfrak{z}}
\newcommand \ga{\mathfrak{a}}
\newcommand \gs{\mathfrak{s}}

\newcommand \Tr{\operatorname{Tr}}
\newcommand \rk{\operatorname{rk}}
\newcommand{\Ric}{\operatorname{Ric}}

\newcommand \Span{\operatorname{Span}}

\theoremstyle{plain}

\newtheorem*{conj*}{Conjecture}
\newtheorem{thm}{Theorem}[section]

\newcommand{\bal}{\begin{aligned}}
\newcommand{\eal}{\end{aligned}}

\usepackage{float}

\theoremstyle{definition}
\newtheorem*{defn*}{Definition}

\theoremstyle{remark}
\newtheorem{rem}[thm]{Remark}
\newtheorem{ex}{Example}

\numberwithin{equation}{section}

\begin{document}	

\title[]{On weakly Einstein Lie groups}

\author{Yunhee Euh}
\address{Department of Mathematics, Sungkyunkwan University, Suwon 16419, Korea}
\email{prettyfish@skku.edu}

\author{Sinhwi Kim}
\address{Department of Mathematics, Sungkyunkwan University, Suwon 16419, Korea}
\email{kimsinhwi@skku.edu}

\author{Yuri Nikolayevsky}
\address{Department of Mathematical and Physical Sciences, La Trobe University, Melbourne, 3086, Australia}
\email{y.nikolayevsky@latrobe.edu.au}

\author{JeongHyeong Park}
\address{Department of Mathematics, Sungkyunkwan University, Suwon 16419, Korea}
\email{parkj@skku.edu}

\thanks {
The research of YE  was supported by Basic Science Research Program through the National Research Foundation of Korea(NRF) funded by the Ministry of Education(Grant number RS-2023-00244736). \\
\indent The research of SK was supported by Basic Science Research Program through the National Research Foundation of Korea(NRF) funded by the Ministry of Education(Grant number RS-2023-00247409). \\
\indent The research of YN was partially supported by ARC Discovery grant DP210100951. \\
\indent The research of JP was supported by the National Research Foundation of Korea(NRF) grant funded by the Korea government(MSIT) (RS-2024-00334956). }	




\begin{abstract} 
A Riemannian manifold is called \emph{weakly Einstein} if the tensor $R_{iabc}R_{j}^{~~abc}$ is a scalar multiple of the metric tensor $g_{ij}$. We consider weakly Einstein Lie groups with a left-invariant metric which are weakly Einstein. We prove that there exist no weakly Einstein non-abelian $2$-step nilpotent Lie groups and no weakly Einstein non-abelian nilpotent Lie groups whose dimension is at most $5$. We also prove that an almost abelian Lie group is weakly Einstein if and only if at the Lie algebra level it is defined by a normal operator whose square is a multiple of the identity. 
\end{abstract}

\maketitle

\section{Introduction}
\label{s:intro}

An $n$-dimensional Riemannian manifold $(M,g)$ is said to be \emph{weakly Einstein} if
    \begin{equation} \label{eq:wE}
        \check{R}_{ij}:=R_{iabc}R_{j}^{~~abc}=\frac{\|R\|^2}{n}g_{ij},
    \end{equation}
where $R$ is the curvature tensor. The definition of weakly Einstein manifolds was introduced by Euh, Park, and Sekigawa in their study of universal curvature identities on 4-dimensional Riemannian manifold \cite{EPS15, EPS13, EPS14}. We mention several known results on weakly Einstein manifolds. In the $3$-dimensional case, by~\cite[Lemma~5]{GHMV}, any weakly Einstein manifold is either of constant curvature, or has its Ricci operator of rank one (in dimension $2$, any Riemannian manifold is weakly Einstein). In dimension $4$, any Einstein manifold is weakly Einstein by \cite{EPS13} (which originally motivated the definition). It is important to emphasize that in general the properties of being Einstein and being weakly Einstein are independent (although in a class of manifolds possessing a specific structure, these two conditions may be equivalent \cite[Corollary 2.3.1]{KPS23}). Moreover, in general the proportionality function $\|R\|^2$ in the definition above does not have to be constant. However, if a manifold is weakly Einstein \emph{and} Einstein (or more generally, if its Ricci tensor is Codazzi), then $\|R\|^2$ must be constant when $n \ne 4$ \cite[Lemma~3.3]{BV}. Some authors include the requirement of not being Einstein in the definition of weakly Einstein manifold. We do not do that and will always say explicitly when the non-Einstein condition is being assumed.

Several facts on weakly Einstein manifolds can be found in \cite[\S 6.55-6.63]{Be}. Conformally flat weakly Einstein manifolds are classified in \cite[Theorem~2]{GHMV}. For results on the weakly Einstein condition in extrinsic geometry we refer the reader to~\cite{KNP} and references therein.

There is a number of results on \emph{homogeneous} weakly Einstein manifolds known in the literature. In~\cite[Example~3.7]{EPS14} the authors gave an example of solvable, almost abelian Lie group of dimension $4$ with a left-invariant metric which is weakly Einstein, but not Einstein. This example is referred in the literature as an \emph{EPS space} (see Theorem~\ref{thm:aa} and Remark~\ref{rem:eps} below). Furthermore, any irreducible symmetric space is weakly Einstein (as the isotropy representation is irreducible); a reducible symmetric space is weakly Einstein if the factors have the same proportionality constant $\|R\|^2/n$. Arias-Marco and Kowalski classified weakly Einstein $4$-dimensional homogeneous Riemannian manifolds~\cite{AK}: they showed that if such manifold is not Einstein, then it must be either an EPS space or a symmetric space (which in dimension $4$ gives the Riemannian product of two surfaces of constant Gauss curvatures $\kappa > 0$ and $-\kappa$).

In this paper we study the weakly Einstein condition for Lie groups with left-invariant metric. Our first main result is motivated by the construction of the EPS spaces; it shows that any weakly Einstein almost abelian Lie group is a higher-dimensional generalization of the EPS space. Recall that a Lie group is called \emph{almost abelian} if it contains an abelian normal subgroup of codimension $1$. A Lie algebra $\gs$ is called almost abelian if it contains a codimension $1$ abelian ideal $\ga$. Clearly any almost abelian Lie algebra is solvable. Given an almost abelian Lie algebra $\gs$ equipped with an inner product $\ip$ we take a unit vector $e \in \gs$ orthogonal to $\ga$ and denote $Q= \ad(e)|_\ga$. We prove the following:

\begin{thm}\label{thm:aa}
Let $G$ be an almost abelian Lie group with a left-invariant metric $g$. The space $(G,g)$ is weakly Einstein if and only if the operator $Q$ is normal and its symmetric part $S$ squares to a multiple of the identity \emph{(}and $(G,g)$ is not Einstein if $S$ by itself is not a multiple of the identity\emph{)}.
\end{thm}

From Theorem~\ref{thm:aa} it is clear that no nilpotent, non-abelian almost abelian Lie group is weakly Einstein. We found no counterexamples to the following conjecture.

\begin{conj*}
  No nilpotent, non-abelian Lie group with a left-invariant metric is weakly Einstein.
\end{conj*}

Note that the analogous fact in the \emph{Einstein} case is well known \cite[Theorem~2.4]{M}. We support the Conjecture by the following two results. 

\begin{thm}\label{thm:2step}
No non-abelian, $2$-step nilpotent Lie group admits a left-invariant, weakly Einstein metric.
\end{thm}

\begin{thm}\label{thm:5nilp}
No non-abelian nilpotent Lie group of dimension at most $5$ admits a left-invariant, weakly Einstein metric.
\end{thm}

In Example~\ref{ex:cossin} we construct a weakly Einstein solvable Lie group whose Lie algebra has a codimension $2$ abelian ideal. The simplest example of a non-solvable weakly Einstein, non-Einstein Lie group would be the Riemannian product of the group $\mathrm{SU}(2)$ with the bi-invariant metric of curvature $\kappa$ and the hyperbolic plane of curvature $-\sqrt{2}\kappa$ viewed as a two-dimensional solvable Lie group.


\section{Preliminaries}
\label{s:prel}

Let $(G,g)$ be a Lie group with a left-invariant metric $g$. Let $\g$ be the Lie algebra of $G$ and denote $\ip$ the inner product on $\g$ (the value of $g$ at the identity of $G$).

By Koszul formula, for the Levi-Civita connection of the metric $g$ we have
 \begin{equation} \label{eq:koszul}
 \<\nabla_UV,W\>=\tfrac12 (\<[U,V],W\>-\<[V,W],U\>+\<[W,U],V\>),
 \end{equation}
for left-invariant vector fields $U,V,W\in\g$. We define the curvature tensor by the standard formula $R(X,Y,U,V)= \<(\nabla_X \nabla_Y - \nabla_Y \nabla_X - \nabla_{[X,Y]})U,V\>$ for left-invariant vector fields $X,Y,U,V\in\g$.

\section{Weakly Einstein solvable Lie groups. Proof of Theorem~\ref{thm:aa}}
\label{s:solv}

In this section, we prove Theorem~\ref{thm:aa}, and then give an example of a solvable weakly Einstein Lie group whose Lie algebra has a codimension $2$ abelian ideal.

\begin{proof}
Let $\gs$ be an almost abelian Lie algebra with the inner product $\ip$, and let $\ga \subset \gs$ be a codimension $1$ abelian ideal. Let $\{e_i\}$ be an orthonormal basis such that $\ga = \Span(e_2, \dots, e_n)$. Denote $Q = \ad(e_1)|_\ga$, and let $S$ and $K$ be the symmetric and the skew-symmetric part of $Q$, respectively, so that $S=\frac{1}{2}(Q+Q^t)$ and $K=\frac{1}{2}(Q-Q^t)$. We note that $Q$ is normal (that is, $[Q,Q^t]=0$) if and only if $[S,K]=0$.

From~\eqref{eq:koszul} we obtain
\begin{equation*}
     \nabla_{e_1}e_1=0, \quad \nabla_{e_1}X = KX, \quad \nabla_X e_1=-SX, \quad \nabla_XY = \<SX,Y\>e_1,
\end{equation*}
for $X, Y \in \ga$. Then the curvature tensor is given by
\begin{equation} \label{eq:aaR}
\begin{gathered}
  R(X,Y)Z = \<SX,Z\>SY - \<SY,Z\>SX, \\ R(e_1,X)Y = -\<PX,Y\>e_1, \quad \text{where } P=[S,K]+S^2,
\end{gathered}
\end{equation}
for $X, Y, Z \in \ga$, and so computing $\check{R}$ by~\eqref{eq:wE} we obtain
\begin{equation*}
  \check{R}(X,Y) = 2\<(P^2 + \Tr(S^2)S^2 - S^4)X,Y\>, \qquad \check{R}(e_1,X) = 0, \qquad \check{R}(e_1,e_1) = 2 \Tr(P^2).
\end{equation*}
Hence the weakly Einstein condition is given by the following equation:
\begin{equation} \label{eq:aawE}
  P^2 + \Tr(S^2)S^2 - S^4 = \Tr(P^2) I_{n-1}.
\end{equation}
Multiplying both sides of \eqref{eq:aawE} by $S^2$ and taking the traces we get $\Tr((P^2-S^4)S^2)=\Tr(P^2-S^4)\Tr(S^2)$. But note that $P^2-S^4 = [S,K]^2+[S,K]S^2+S^2[S,K]$ and that $\Tr([S,K]S^m)=\Tr(S^m[S,K])=0$, for any $m \ge 0$. Hence we obtain $\Tr([S,K]^2 S^2) = \Tr([S,K]^2) \Tr(S^2)$. Assuming the matrix $S$ to be diagonal, with the diagonal entries $a_i$ such that $a_1^2 \ge a_2^2 \ge \ldots \ge a_{n-1}^2$, and $d_i, \, i=1, \ldots, n-1$, to be the diagonal entries of $[S,K]^2$ (note that $d_i \ge 0$) we obtain $\sum_i (a_i^2 (\sum_j d_j - d_i)) = 0$, which is only possible when either all $a_i$ are zero (and so $S=0$), or all $d_i$ are zero (and so $[S,K]=0$), or $a_1 \ne 0, \, d_1 \ne 0$ and $a_i =d_i=0$ for $i \ge 2$ (which may not occur as $\rk [S,K]^2$ must be even). Thus $[S,K] = 0$ which proves that $Q$ is normal.

Now in~\eqref{eq:aawE} we have $P=S^2$ and so $\Tr(S^2)S^2 = \Tr(S^4) I_{n-1}$ which is equivalent to the fact that $S^2 = a^2 I_{n-1}$ for some $a \ge 0$.

It remains to note that by~\eqref{eq:aaR}, the Ricci operator is given by $\Ric e_1 = -\Tr(S^2)e_1$ and $\Ric X = -(S^2+\Tr(S) S-\left<SX,X\right>S)X$ for every unit vector $X \in \ga$, which shows that the left-invariant metric is Einstein if and only if $S$ is a multiple of the identity (which is a well-known fact).
\end{proof}

\begin{rem} \label{rem:eps}
  Note that the weakly Einstein solvmanifolds in Theorem~\ref{thm:aa} can be viewed as higher-dimensional generalizations of the EPS spaces introduced in~\cite[Example~3.7]{EPS14}. Indeed, given a normal operator $Q$ acting on $\ga$ whose symmetric part squares to $a^2 I_{n-1}$, one can choose an orthonormal basis for $\ga$ relative to which the matrix $Q$ has a block-diagonal form, with all the diagonal blocks being either of size $1 \times 1$ and containing $a$ or $-a$, or of size $2 \times 2$ and having the form $\pm \left(\begin{smallmatrix} a & b_j \\ -b_j & a \end{smallmatrix} \right)$, where $b_j$ are arbitrary nonzero numbers.

  We also note the following. In the above notation, define the metric solvable, almost abelian Lie algebra $\gs'$ on the same Euclidean space $\br e_1 \oplus \ga$, with $\ga$ being an abelian ideal and with $\ad(e_1)|_\ga = S$. Although $\gs$ and $\gs'$ are not isomorphic when $K \ne 0$, the condition that $Q$ is normal implies that the corresponding solvmanifolds are isometric (see~\cite[Theorem~3.3]{A}, \cite[Lemma~4.2]{AW}, or \cite[Proposition~2.5]{H}). It follows that the weakly Einstein almost abelian solvmanifold in Theorem~\ref{thm:aa} can be viewed as a certain homogeneous hypersurface in the product of two hyperbolic spaces of the same curvature (but of arbitrary dimensions).
\end{rem}

We now give an example of a solvable weakly Einstein Lie group whose Lie algebra has a codimension $2$ abelian ideal.

\begin{ex} \label{ex:cossin}
Let $G$ be a Lie group with the Lie algebra $\gs$ whose nonzero Lie brackets relative to an orthonormal basis $\{f_1, f_2, e_1, \ldots, e_m\}, \; m \ge 3$, are given by
\begin{equation*}
[f_1,e_j]=\cos\bigl(\tfrac{2\pi j}{m}\bigr) \, e_j, \qquad  [f_2,e_j]=\sin\bigl(\tfrac{2\pi j}{m}\bigr)\, e_j.
\end{equation*}
For $j=1, \ldots, m$, we denote $\lambda_{1j}=\cos\left(\frac{2\pi j}{m}\right)$ and $\lambda_{2j}=\sin\left(\frac{2\pi j}{m}\right)$.
From~\eqref{eq:koszul} we have the following for the Levi-Civita connection:
\begin{equation*}
\nabla_{f_a}f_b=\nabla_{f_a}e_j=0, \quad \nabla_{e_j}f_a=-\lambda_{aj}e_j, \quad \nabla_{e_j} e_k=\delta_{jk}(\lambda_{1j}f_1+\lambda_{2j}f_2),
\end{equation*}
for $j,k=1, \ldots, m$ and $a,b =1,2$. This gives the following (potentially) nonzero components of the curvature tensor:
\begin{equation*}
    R(f_a,e_j, e_j, f_b) =-\lambda_{aj}\lambda_{bj}, \qquad     R(e_j,e_k,e_k,e_j)=-(\lambda_{1j}\lambda_{1k}+\lambda_{2j}\lambda_{2k}),
\end{equation*}
from which $\check{R}(f_a,e_j) = 0$ and
\begin{align*}
\check{R}(f_a,f_b)&=2\sum\nolimits_{h=1}^m \lambda_{ah}\lambda_{bh}(\lambda_{1h}^2+\lambda_{2h}^2) = \delta_{ab} m , \\
\check{R}(e_j,e_k)&=2\delta_{jk} \sum\nolimits_{h=1}^m (\lambda_{1j}\lambda_{1h}+\lambda_{2j}\lambda_{2h})^2 = \delta_{jk}m.
\end{align*}
It follows that $(G,g)$ is weakly Einstein (and it is easy to see that it is not Einstein).
\end{ex}

\section{2-step nilpotent groups. Proof of Theorem~\ref{thm:2step}}
\label{s:2step}

Let $\gn$ be a 2-step nilpotent Lie algebra with an inner product $\ip$. Let $N$ be the associated Lie group with the left-invariant metric induced by $\ip$. Consider the orthogonal decomposition $\gn = \gv \oplus \gz$, where $\gz$ is the derived algebra of $\gn$. Denote $d_\gz = \dim \gz$ and $d_\gv = \dim \gv$.

For every $A\in\gz$, we define $J_A:\gv\rightarrow\gv$ by
\begin{equation*}
\<J_AX,Y\>=\<[X,Y],A\>,
\end{equation*}
for $X,Y\in\gv$. Note that for all $A \in \gz$, the operator $J_A$ on $\gv$ is skew-symmetric, and that the map from $\gz$ to $\so(\gv)$ defined by $A \mapsto J_A$ is injective.

From~\eqref{eq:koszul} we obtain
\begin{eqnarray*}
\nabla_AX=\nabla_XA=-\tfrac{1}{2}J_AX,\qquad
\nabla_XY=\tfrac{1}{2}[X,Y], \qquad \nabla_AB = 0,
\end{eqnarray*}
for (left-invariant vector fields) $X, Y \in \gv$ and $A, B \in \gz$. These give the following well-known formulas for the curvature tensor:
\begin{gather*}
R(A,B)C=0, \qquad R(X,A)B = -\tfrac14 J_AJ_BX, \qquad R(A,B)X =\tfrac14 [J_A,J_B](X),\\
R(X,Y)A =-\tfrac14 [X,J_AY]+\tfrac14 [Y,J_AX], \qquad R(X,A)Y =-\tfrac14 [X,J_AY], \\
R(X,Y)Z =\tfrac12 J_{[X,Y]}Z-\tfrac14 J_{[Y,Z]}X-\tfrac14 J_{[Z,X]}Y,
\end{gather*}
where $X,Y,Z\in\gv$ and $A,B,C\in\gz$.

Choose an orthonormal basis $\{e_1, \dots, e_{d_\gv + d_\gz}\}$ for $\gn$ in such that $\Span(e_1, \dots, e_{d\gz}) = \gz$, and adopt the following index convention: $i,j \in \{d_\gz+1, \dots, d_\gv+d_\gz\}$ and $a,b,c \in \{1, \dots, d_\gz\}$. We abbreviate $J_{e_a}$ to $J_a$, and define the symmetric operators $S_{ab}$ on $\gv$ by
\begin{equation*}
S_{ab}= -\tfrac12 (J_aJ_b+J_bJ_a), \qquad \text{for } a, b = 1, \dots, d_\gz.
\end{equation*}
Note that all the operators $S_{aa}$ are non-negative definite and nonzero.

A straightforward computation by formula~\eqref{eq:wE} gives the following: 
\begin{equation}\label{eq:2stepcheckcR}
\begin{split}
\check{R}_{ia}&=0,\\
\check{R}_{ab}&= \frac18 \sum_c \Tr (2J_aJ_c^2J_b-J_bJ_cJ_aJ_c) = \frac18 \sum_c \Tr (3S_{ab}S_{cc}-2S_{ac}S_{bc}),\\
\check{R}_{ij}&= \frac{1}{16}\sum_{a,b}\<(-[J_a,J_b]^2 +2 J_aJ_b^2J_a +6 \Tr (J_aJ_b) J_aJ_b + 6 (J_aJ_b)^2)e_i,e_j\> \\
&=\frac18 \sum_{a,b}\<(4S_{ab}^2 + 3 \Tr (S_{ab})S_{ab})e_i,e_j\>.
\end{split}
\end{equation}

\begin{proof}[Proof of Theorem~\ref{thm:2step}] Suppose a $2$-step nilpotent Lie group $G$ with a left-invariant metric is weakly Einstein. In the above notation, we can assume that the orthonormal basis $\{e_a\}, \, a= 1, \dots, d_\gz$, is chosen in such a way that $\Tr (S_{ab}) = 0$ for $a \ne b$ and that $\Tr (S_{11}) \le \Tr (S_{aa})$, for all $a, b = 1, \dots, d_\gz$ (note that $\Tr (S_{11}) > 0$).

By~\eqref{eq:2stepcheckcR}, the weakly Einstein condition gives
\begin{eqnarray*}
\Tr \Big(S_{aa}^2+3\sum_{c\neq a}S_{aa}S_{cc}-2\sum_{c\neq a}S_{ac}^2\Big)=8C,\qquad
4\sum_{a,b}S_{ab}^2+3\sum_a\Tr (S_{aa})S_{aa}=8C \, I_{d_\gv},
\end{eqnarray*}
for some $C > 0$. Take $a=1$ in the first equation and multiply both of its sides by $\Tr(S_{11})$. Then multiply both sides of the second equation by $S_{11}$ and take the traces. Subtracting the resulting two equations from one another we obtain
\begin{multline*}
  2 \Tr (S_{11})\Tr (S_{11}^2)+3 \sum_{c\neq 1}(\Tr (S_{cc})-\Tr (S_{11}))\Tr (S_{cc}S_{11}) \\
  +4\Tr \Big(S_{11}\sum_{a,b}S_{ab}^2\Big)+2\sum_{c\neq 1}\Tr (S_{11})\Tr (S_{1c}^2)=0.
\end{multline*}
But now every term on the left-hand side is nonnegative, and the first one is strictly positive, which gives the desired contradiction.
\end{proof}

\section{$5$-dimensional nilpotent Lie groups. Proof of Theorem~\ref{thm:5nilp}}
\label{s:5nilp}

Suppose a non-abelian, nilpotent Lie group $N$ of dimension $n \le 5$ with a left-invariant metric $g$ is weakly Einstein. Let $(\gn, \ip)$ be the corresponding metric nilpotent Lie algebra. By~\cite{AK} (or by using the fact that any nilpotent algebra of dimension at most $4$ is either $2$-step nilpotent or almost abelian and then applying Theorem~\ref{thm:aa} and Theorem~\ref{thm:2step}), we can assume that $n = 5$. We will also assume that the nilpotency class of $\gn$ is at least $3$ (by Theorem~\ref{thm:2step}).

The classification of nilpotent Lie algebras of dimension $5$ is well known in the literature. We will use the classification and the notation of \cite[Section~4]{Gr}, according to which a nilpotent Lie algebra of dimension $5$ with nilpotency class at least $3$ is isomorphic to one of the following five Lie algebras given by their nonzero Lie brackets relative to a basis $\{F_1, F_2, F_3, F_4, F_5\}$:
\begin{itemize}
    \item $\gl_{4,3} \times \br$: $[F_1,F_2]=F_3$, $[F_1,F_3]=F_4$
    \item $\gl_{5,5}$: $[F_1, F_2]=F_4$, $[F_1, F_3]=F_5$, $[F_2, F_4]=F_5$.
    \item $\gl_{5,6}$: $[F_1, F_2]=F_3$, $[F_1, F_3]=F_4$, $[F_1, F_4]=F_5$, $[F_2, F_3]=F_5$.
    \item $\gl_{5,7}$: $[F_1, F_2]=F_3$, $[F_1, F_3]=F_4$, $[F_1, F_4]=F_5$.
    \item $\gl_{5,9}$: $[F_1, F_2]=F_3$, $[F_1, F_3]=F_4$,  $[F_2, F_3]=F_5$.
\end{itemize}
Note that the algebras $\gl_{4,3} \times \br$ and $\gl_{5,7}$ are almost abelian (with $\Span(F_2, F_3, F_4, F_5)$ being a codimension $1$ abelian ideal), and so are prohibited by Theorem~\ref{thm:aa}. It follows that $\gn$ is isomorphic to one of the algebras $\gl_{5,5}, \gl_{5,6}$ or $\gl_{5,9}$. According to~\cite[Theorems~9, 11, 14]{FN}, for the metric Lie algebra $(\gn, \ip)$ one can choose an orthonormal basis $\{E_i\}$ such that the Lie brackets are given by
\begin{equation}\label{eq:L56_lie}
\begin{gathered}
    [E_1, E_2]=aE_3+bE_4+cE_5, \quad [E_1, E_3]=dE_4+fE_5, \\
    [E_1, E_4]=gE_5, \quad [E_2, E_3]=hE_5,
\end{gathered}
\end{equation}
where for $\gn \cong \gl_{5,5}$ we have $b, g, h \ne 0 = a$, for $\gn \cong \gl_{5,9}$ we have $a, h, d \ne 0 = f = g$, and for $\gn \cong \gl_{5,6}$ we have $a,d,g,h\ne0$. Note that $h \ne 0$ in all three cases.

The rest of the proof requires some direct (computer aided) computation. Using~\eqref{eq:koszul} we find the components of the Levi-Civita connection relative to the basis $\{E_i\}$, and then the components of the curvature tensor (those which are not listed, up to $\mathbb{Z}_2$-symmetry,  are zeros):
\begin{equation}
\begin{aligned}
    \label{eq:L56_curvature}
    R_{1212}&=\tfrac14(3a^2+3b^2+3c^2),\quad R_{1213}=\tfrac34(bd+cf),\quad R_{1214}=\tfrac14(ad+3cg),\\
    R_{1215}&=\tfrac14(af+bg),\quad R_{1223}=\tfrac34 ch,\quad R_{1225}=\tfrac14 ah,\quad R_{1234}=-\tfrac14 gh,\\
    R_{1245}&=\tfrac14 dh,\quad R_{1313}=\tfrac14(-a^2+3d^2+3f^2),\quad R_{1314}=\tfrac14(-ab+3fg),\\
    R_{1315}&=\tfrac14(-ac+dg),\quad R_{1323}=\tfrac34 fh,\quad R_{1324}=\tfrac14 gh,\quad R_{1335}=-\tfrac14 ah,\\
    R_{1345}&=-\tfrac14 bh,\quad R_{1414}=\tfrac14(-b^2-d^2+3g^2),\quad R_{1415}=-\tfrac14(bc+df),\\
    R_{1423}&=\tfrac12 gh,\quad R_{1515}=-\tfrac14(c^2+f^2+g^2),\quad R_{1524}=-\tfrac14 dh,\quad R_{1525}=-\tfrac14 fh,\\
    R_{1534}&=\tfrac14 bh,\quad R_{1535}=\tfrac14 ch,\quad R_{2323}=\tfrac14(-a^2+3h^2),\quad R_{2324}=-\tfrac14 ab,\\
    R_{2325}&=-\tfrac14 ac,\quad R_{2334}=\tfrac14 ad,\quad R_{2335}=\tfrac14 af,\quad R_{2345}=\tfrac14(-cd+bf),\\
    R_{2424}&=-\tfrac14 b^2,\quad R_{2425}=-\tfrac14 bc,\quad R_{2434}=-\tfrac14 bd,\quad R_{2435}=\tfrac14(-cd+ag),\\
    R_{2445}&=\tfrac14 bg,\quad R_{2525}=-\tfrac14(c^2+h^2),\quad R_{2534}=\tfrac14(-bf+ag),\quad R_{2535}=-\tfrac14 cf,\\
    R_{2545}&=-\tfrac14 cg,\quad R_{3434}=-\tfrac14 d^2,\quad R_{3435}=-\tfrac14 df,\quad R_{3445}=\tfrac14 dg,\\
    R_{3535}&=-\tfrac14(f^2+h^2),\quad R_{3545}=-\tfrac14 fg,\quad R_{4545}=-\tfrac14 g^2.
\end{aligned}
\end{equation}
The weakly Einstein condition is strongly over-determined: computing $\check{R}_{ij}$ (by~\eqref{eq:wE}) we obtain $14$ homogeneous equations of degree $4$ for seven unknowns $a,b,c,d,f,g$ and $h$.

We outline the steps of the proof (which can be easily verified) without giving explicit expressions. 

First suppose that $a=0$. Up to a rotation of the basis in the $(E_2,E_3)$-plane, we can assume that $b=0$ (and then $d \ne 0$, for otherwise $\gn$ is $2$-step nilpotent). Then {{$\check{R}_{45} = \frac14 df(c^2 + d^2 + f^2 + h^2)=0$}} and $\check{R}_{35} = \frac14 dg (2d^2 - f^2)=0$, and so $f=g=0$. But then {{$\check{R}_{22} = \frac14 (c^2 + h^2)(5 c^2 + d^2 + 5 h^2)$}} and $\check{R}_{55} = \frac14 (c^2 + h^2)(c^2 + d^2 + h^2)$, which cannot be equal as $h \ne 0$. Thus $a \ne 0$.

Next assume that $d=0$. Up to simultaneous rotations in the $(E_1,E_2)$- and $(E_3,E_4)$-planes, we can make $f=0$. But then $c=0$ from the equation $\check{R}_{35} = \frac14 a c (a^2 + b^2 + c^2 + g^2)=0$, which gives $\check{R}_{15} = -\frac12  a h (a^2 + b^2) \ne 0$ as $a, h \ne 0$, a contradiction. So $d \ne 0$.

Now suppose that $g=0$. Then up to simultaneous rotations of the basis in the $(E_1,E_2)$- and $(E_4,E_5)$-planes, we can make $f=0$. Then $bc=0$ from {{$\check{R}_{45} = \frac14 b c (a^2 + b^2 + c^2 + d^2 + h^2) = 0$}}. Moreover, {{$\check{R}_{44} = \frac14 (b^2 + d^2) (a^2 + b^2 + c^2 + d^2 + h^2)$}} and {{$\check{R}_{55} = \frac14 (c^2 + h^2) (a^2 + b^2 + c^2 + d^2 + h^2)$}} which gives {{$b^2 - c^2 + d^2 - h^2 = 0$}}. Then from equations {{$\check{R}_{15} = -\frac14  a h (2 a^2 + 2 b^2 - c^2 - d^2) = 0$}} and $\check{R}_{24} = \frac14 a d (2 a^2 - b^2 + 2 c^2 - h^2)=0$ we obtain $b=c=0$ and $d^2 = h^2 = 2 a^2$. But then the matrix of $\check{R}$ is given by $\frac14 a^4 \mathrm{diag}(27,27,33,10,10)$, a contradiction.

Therefore we have $a,d,g,h \ne 0$. As $a, d, h \ne 0$, we can express $f$ in terms of $b, d$ and $c$ from {{$\check{R}_{14} = \frac14  a h (3 b c - d f) = 0$}} and then substitute back into $\check{R}$. We denote $\check{R}^{(1)}$ the resulting matrix multiplied by $4d^4$; its entries are homogeneous polynomials of degree $8$. We have {{$\check{R}^{(1)}_{15}=-a d^4 h  (2 a^2 + 2 b^2 - c^2 - d^2)$}}, and so $c^2=2 a^2 + 2 b^2 - d^2$. Substituting this into all the entries of $\check{R}^{(1)}$ we obtain the matrix $\check{R}^{(2)}$ such that the degree of each of its entries in $c$ is at most $1$. We have $\check{R}^{(2)}_{25}=-6 b d^4 g (a^2 + b^2 - d^2)$.

We first assume that $b \ne 0$. Then from $\check{R}^{(2)}_{25}=0$ we have $b^2 = d^2 - a^2$, and we can substitute this expression in the rest of the entries to obtain the matrix $\check{R}^{(3)}$ whose entries have degree at most $1$ in each of the variables $b$ and $c$. Now $\check{R}^{(3)}_{45} = -4 b d^4  (9 a^2 c + a d g - 12 c d^2  - c h^2)$. We solve for $g$ and then substitute into the other entries of the matrix $\check{R}^{(3)}$, multiply the resulting matrix by $a^4 d^{-3}$ and reduce modulo the relations $c^2 = 2 a^2 + 2 b^2 - d^2$ and $b^2 =  d^2 - a^2$ to obtain the matrix $\check{R}^{(4)}$ whose entries are homogeneous polynomials of degree $9$ in the variables {{$(a,b,c,d,h)$}} having degree at most $1$ in $b$. But then the only entries of $\check{R}^{(4)}$ which contain $b$ are $\check{R}^{(4)}_{12}, \, \check{R}^{(4)}_{23}$ and $\check{R}^{(4)}_{34}$, and in all three $b$ is a factor. It follows that if $\check{R}^{(4)}$ is a multiple of the identity matrix for some values of $(a,b,d,c,h)$, then it is also (the same) multiple of the identity matrix, if we take the same values for $a,d,c$ and $h$ and replace $b$ with $0$. Therefore we can assume that $b=0$ in the matrix $\check{R}^{(2)}$.

Taking $b= 0$ in $\check{R}^{(2)}$ and dividing all the entries by $d^4$ we obtain the matrix $\check{R}^{(5)}$ whose entries are homogeneous polynomials of degree $4$. From the equation $a \check{R}^{(5)}_{13} + 4 h \check{R}^{(5)}_{35}=0$ we obtain $c= -d g (a^2 - 8 d^2)/(a (2 d^2 + g^2 + 5 h^2))$. Substituting this in the other entries and multiplying the resulting matrix by $a (2 d^2 + g^2 + 5 h^2)$ we obtain the matrix $\check{R}^{(6)}$ with $\check{R}^{(6)}_{35} = (-3 a^4 + 26 a^2 d^2 - a^2 g^2 - 12 d^4 + 10 d^2 g^2 + 10 d^2 h^2) a d g$. Solving for $h^2$ and substituting into the rest of the matrix and multiplying by $200 d^6$ to clear the denominators we obtain a matrix all of whose entries are divisible by $a(a^2 - 8 d^2)$. Dividing by this expression (it cannot be zero as $\check{R}^{(6)} \ne 0$) we get the matrix $\check{R}^{(7)}$ whose entries are homogeneous polynomials of degree $10$ in the variables $a, d$ and $g$. Then the expressions for $(10 a d^3)^{-1}\check{R}^{(7)}_{24}$, for $(20 d^2)^{-1} (\check{R}^{(7)}_{22} - \check{R}^{(7)}_{33})$ and for $(10 d^2)^{-1} (\check{R}^{(7)}_{11} - \check{R}^{(7)}_{44})$ give three homogeneous polynomials $P_i,\, i=1,2,3$, of degrees $3, 4$ and $4$, respectively, in the variables $a^2,d^2$ and $g^2$. Computing the resultant $P_1$ and $P_2$ relative to $d^2$ and the resultant $P_1$ and $P_3$ relative to $d^2$ we obtain two homogeneous polynomials in $a^2$ and $g^2$ whose resultant relative to $g^2$ is a nonzero constant times $a^{50}$. As $a \ne 0$, we arrive at a contradiction.

This completes the proof of Theorem~\ref{thm:5nilp}.

\end{document}